\documentclass[a4paper,reqno]{amsart}
\pdfoutput=1
\usepackage[T1]{fontenc}
\usepackage{lmodern}

\usepackage{amsmath,amsfonts,amssymb,amsthm,stmaryrd,bbm,graphicx,mathtools,enumerate}
\usepackage{mathrsfs}
\usepackage[T1]{fontenc} 
\usepackage[latin1]{inputenc}
\usepackage[english]{babel}
\usepackage[tracking,spacing,kerning,babel]{microtype}
\usepackage{wasysym} 
\usepackage{color}
\usepackage{xcolor}
\usepackage{framed} 
\usepackage{xparse}
\usepackage{mathtools}
\usepackage{relsize}
\usepackage{soul}

\usepackage{wrapfig} 

\usepackage[final]{hyperref}   
\hypersetup{
    linktoc=page,
    linkcolor=red,          
    citecolor=blue,        
    filecolor=blue,      
    urlcolor=cyan,
   colorlinks=true           
}

\newcommand{\E}{\mathbb{E}}

    \newcommand{\Prb}{\mathbb{P}}

		\newcommand{\cE}{\mathcal{E}}
		\newcommand{\cF}{\mathcal{F}}

		\newcommand{\f}{\widetilde{\Phi}}

	\newcommand{\sR}{\mathbb{R}}

		\newcommand{\rN}{\mathrm{N}}

				\DeclareMathOperator{\cyl}{cyl}
				\DeclareMathOperator{\Var}{Var}

    \newcommand{\sZ}{\mathbb{Z}}
     \newcommand{\Z}{\mathbb{Z}}
     \newcommand{\R}{\mathbb{R}}

    \newcommand{\ep}{\varepsilon}

    \newcommand{\ind}{\mathbf{1}}

\theoremstyle{plain}   
\newtheorem{thm}{Theorem}
\numberwithin{thm}{section}
\newtheorem{lem}[thm]{Lemma}
\newtheorem{cor}[thm]{Corollary}
\newtheorem{prop}[thm]{Proposition}

\newtheorem{opn}{Open question}

\theoremstyle{remark}
\newtheorem{rk}{Remark}

\definecolor{darkgreen}{rgb}{0,0.6,0.05}

\title[Superconcentration for minimal surfaces]{Superconcentration for minimal surfaces in \\ first passage percolation and  \\ disordered Ising ferromagnets}

\author{Barbara Dembin  \;\;\;\;\;\;\;\;  Christophe Garban}

\address
{D-MATH, ETH Z\"urich, R\"amistrasse 101, 8092 Z\"urich, Switzerland }
\email{barbara.dembin@math.ethz.ch}

\address
{Universit\'e Claude Bernard Lyon 1, CNRS UMR 5208, Institut Camille Jordan, 69622 Villeurbanne, France \,, Institut Universitaire de France (IUF) and Universit\'e de Gen\`eve (Unige)}
\email{garban@math.univ-lyon1.fr}

\date{}
\usepackage{pgffor}
\begin{document}
\maketitle

\begin{abstract}
We consider the standard first passage percolation model on $\mathbb Z^ d$ with a distribution $G$ taking two values $0<a<b$. We study the maximal flow through the cylinder $[0,n]^ {d-1}\times [0,hn]$ between its top and bottom as well as its associated minimal surface(s). We prove that the variance of the maximal flow is superconcentrated, i.e. in  $O(\frac {n^{d-1}} {\log n})$, for $h\geq h_0$ (for a large enough constant $h_0=h_0(a,b)$). 

Equivalently, we obtain that the ground state energy of a disordered Ising ferromagnet in a cylinder $[0,n]^ {d-1}\times [0,hn]$ is superconcentrated when opposite boundary conditions are applied at the top and bottom faces and for a large enough constant $h\geq h_0$ (which depends on the law of the coupling constants).

 Our proof is inspired by the proof of Benjamini--Kalai--Schramm \cite{BKS}. Yet, one major difficulty in this setting is to control the influence of the edges since the averaging trick used in \cite{BKS} fails for surfaces. 

Of independent interest, we prove that minimal surfaces (in the present discrete setting) cannot have long thin chimneys.
\end{abstract}

\section{Introduction}

\subsection{Context and main results}

We focus in this paper on the fluctuations of the maximal flow (or equivalently of the minimal surface of the dual problem) through a cylinder in $\Z^d$ of the form $[0,n]^{d-1}\times[0, H]$, where the vertical height $H$ will be through most of this text of order $h n$. It is defined informally as follows (see Subsection \ref{s.serious} below for a more formal definition). Each non-oriented edge $e$ inside $[0,n]^{d-1}\times[0,h n]$ carries an i.i.d capacity $t(e)$ whose distribution takes two values $0<a<b$. Without much loss of generality, one can think of $t(e)\in \{1,2\}$ with equal probability.  The (vertical) maximum flow through this cylinder is informally the maximum amount of {\em water} which can be injected at the bottom, say, of the cylinder so that it can flow upwards in such a way that the amount of water flowing through any given edge $e$ is less or equal than $t(e)$. Let us denote this maximal flow by $\Phi=\Phi([0,n]^{d-1}\times\{0\}, H)$. By max-flow/min-cut principle, it is well-known that this maximal flow can be computed by minimizing the capacity over all possible cut-sets. I.e,
\begin{align*}\label{}
\Phi = \min_{E} \left\{\sum_{e\in E} t(e) \right\}\,,
\end{align*}
where the mimimum is taken over all cut-sets $E$ which separate the bottom $[0,n]^{d-1}\times \{0\}$ from the top $[0,n]^{d-1}\times \{H\}$. There may be several such minimizing  cut-sets $E$ and by duality each of those correspond to a minimal surface embedded in $\R^d$ (see Figure \ref{fig8}). 

\smallskip

In dimension $d=2$, the minimal cut-sets in $[0,n]\times[0,H]$ correspond to geodesics on the dual graph $(\Z^2)^*=\Z^2+(\tfrac 1 2, \tfrac 1 2)$ which connect the left and right boundaries of the rectangle.  The maximal flow can then be studied as a random metric problem in this special case and much is known about fluctuations, large-deviations etc. in this case.  Let us mention in particular the breakthrough work by Benjamini-Kalai-Schramm \cite{BKS} which implies in the present setting that $\Var[\Phi([0,n]\times \{0\}, H)]=O(\frac  n {\log n} )$ as long as $H=\Omega(n^\epsilon)$. Furthermore, in this $d=2$ case, the fluctuations are believed to be described as $n\to \infty$ by the {\em KPZ universality class} (in particular it is conjectured that $\Var[\Phi]\asymp n^{2/3}$, see for example \cite{johansson2000shape} where this is proved for directed last-passage percolation). 

\smallskip

In higher dimensions $d\geq 3$, the problem may no longer be formulated in terms of geodesics and is expressed instead in terms of minimal surfaces (of co-dimension 1). 
The analysis of such maximal flows/minimal surfaces in $d\geq 3$ was first considered in the seminal paper by Kesten for $d=3$: {\em Surfaces with minimal random weights and maximal flows: a higher dimensional version of first-passage percolation} (\cite{Kesten:flows}) where he obtained a law of large numbers for $\Phi$ as well as some large deviations estimates. Since the work \cite{Kesten:flows}, there has been a lot of activity on the analysis of the maximal flow $\Phi$: Kesten's results were extended by Zhang  \cite{Zhang2017} to any dimensions, and by Rossignol--Th\'eret in \cite{RossignolTheret08b} to any dimensions for tilted flat cylinders (with height $H=o(n)$). Cerf--Th\'eret  proved a law of large number for more general domains in \cite{CerfTheret09geoc}. They later studied the speed of upper and lower large deviations in \cite{CerfTheret09infc,CerfTheret09supc}.  Interestingly, upper large deviations are in $n^{d}$ while lower large deviations are in $n^{d-1}$.  In \cite{dembintheretcutsetld,Dembintheretldmaxflow}, Dembin--Th\'eret proved upper and lower large deviations principle for the maximal flow in general domains.

\medskip

Let us now introduce another setting where minimal surfaces appear in the same fashion. Consider the disordered Ising ferromagnet in  $[0,n]^{d-1}\times[0,h n]$ with opposite boundary conditions applied at the top and the bottom.
Each non-oriented edge $e$ inside $[0,n]^{d-1}\times[0,h n]$ carries an i.i.d coupling constant $J_e$ whose distribution takes two values $0<a<b$. For a configuration $\sigma\in\{-1,1\}^{[0,n]^{d-1}\times[0,h n]\cap \Z^d}$, its associated energy is 
\[H(\sigma)= - \sum_{e=\{x,y\}}J_e\sigma_x\sigma_y.\]
One can check that the ground state energy (i.e. the minimal energy) corresponds to $\Phi$ and the corresponding minimal surface corresponds to the interface of a ground state (i.e. a configuration achieving the minimal energy). This connection was mentioned for example in Licea--Newman \cite{LiceaNewman}.

\smallskip

To our knowledge, prior to this work, nothing was known on the fluctuations of $\Phi=\Phi([0,n]^{d-1}\times[0, H])$ (besides the easy upper bound $\Var[\Phi]= O(n^{d-1})$). As we shall explain further in the next subsection, this may be due to the following reason. A crucial step in the proof of Benjamini-Kalai-Schramm in \cite{BKS} is based on a beautiful averaging trick which no longer works with minimal surfaces. 

Our main result can be stated as follows. 
\begin{thm}\label{thm:main} For any $d\geq 2$ and any distribution $G$ on $0<a<b$, there exist $C>0$ and $h_0>0$, such that for any $n\ge1$ and $H\ge h_0 n$, we have
\[\Var (\Phi([0,n]^{d-1}\times \{0\}, H )\le C \frac{n^{d-1}}{\log n}\,.\]
\end{thm}

As it has been identified in the seminal work by Chatterjee \cite{chatterjee2014superconcentration}, a variance of order $O(\frac{n^{d-1}}{\log n})$ versus a variance of order $\Omega(n^{d-1})$ induces a completely different behaviour of minimal cut-sets under small random perturbations of the capacities $\{t(e)\}_e$. Indeed, a variance negligible w.r.t $n^{d-1}$  corresponds to the phenomenon of {\em superconcentration} (\cite{chatterjee2014superconcentration}) and it implies a certain {\em chaoticity} property for the minimal cut-sets.  We shall illustrate this in Corollary \ref{c.chaotic} where we will rely on a mild extension of a very useful identity from \cite{tassion2020noise}.  See also the recent work of Chatterjee \cite{chatterjee2023spin} which analyzed the groundstate of an Ising model with non-ferromagnetic disordered coupling constants.

\smallskip

We complete our analysis of the fluctuations of $\Phi=\Phi([0,n]^{d-1}\times \{0\}, H )$ by the following easier lower bound on the variance. Its proof in Section \ref{s.thm2} will rely on the martingale decomposition method from  Newman--Piza \cite{newman1995divergence}.
\begin{thm}\label{thm:main2} Let $G$ be a distribution on $\{a,b\}$ such that $G(\{b\})>p_c$, where $p_c$ is the critical parameter for Bernoulli bond percolation on $(\sZ^d,\E^d)$. There exists a constant $c=c(G)>0$ such that for all $n,H\ge 1$, we have
\[\Var (\Phi([0,n]^{d-1}\times\{0\}, H))\ge c\frac{n^{d-1}}{H}\,.\]
\end{thm}

\smallskip
\smallskip

We now  introduce a slightly different model for which a greatly simplified version of our proof  also implies superconcentration (see Remark \ref{r.easier} below). In the same cylinder $[0,n]^{d-1} \times [0,H]$, we now assign i.i.d weights $\{t(x)\}$ to the vertices of the cylinder, again with a distribution $G$ on $0<a<b$.  We consider the following minimal weight
\begin{align*}\label{}
\Psi_{\mathrm{Lip}}=\Psi_{\mathrm{Lip}}([0,n]^{d-1}\times \{0\}, H) := \min_{\psi} \left\{ \sum_{u\in [0,n]^{d-1}}  t( u,\psi(u)) \right\}\,,
\end{align*}
where the minimum is taken over all 1-Lipschitz functions $\psi : [0,n]^{d-1} \to \{0,1,\ldots,H\}$ (i.e. such that $|\psi_i -\psi_j|\leq 1$ for any $i \sim j$ in $[0,n]^{d-1}$). We obtain in this setting the analog of Theorem \ref{thm:main}. 
\begin{thm}\label{thm:psi} There exist $C,c>0$ and $h_0>0$, both depending on $0<a<b$, such that for any $n\ge1$ and $H\ge h_0 n$, we have
\[ \left(c\frac{n^{d-1}}{H} \le \right) \Var (\Psi_{\mathrm{Lip}}([0,n]^{d-1}\times \{0\}, H )\le C \frac{n^{d-1}}{\log n}\,.\]
\end{thm}

To conclude this introduction, we wish to emphasise that if minimal surfaces happen to be anchored at some deterministic curve along  the boundary of the cylinder, then we expect a completely different scenario for their fluctuations in large enough dimensions $d$. We discuss two possible such situations:
\begin{enumerate}
\item Instead of considering the maximum flow $\Phi$ from the bottom $[0,n]^{d-1}\times \{0\}$ to the top $[0,n]^{d-1}\times \{H\}$, let us consider  the maximal flow $\tau([0,n]^{d-1}\times\{0\},H)$ between the bottom half and the top half of the cylinder, (i.e. between $\partial ([0,n]^{d-1}\times [0,H]) \cap \{x\in \R^d, x\cdot \mathbf e_d < \tfrac H 2]\}$ and $ \partial ([0,n]^{d-1}\times [0,H]) \cap \{x\in \R^d, x\cdot \mathbf e_d > \tfrac H 2]\}$). Then, the associated minimal surfaces are anchored in the boundary of the meridian plane of the cylinder $[0,n]^{d-1} \times \{\tfrac H 2 \}$. For a formal definition, we refer to \cite{RossignolTheret08b}. In high dimensions, by analogy to other models of surface (see in particular \cite{peled2017high}), we expect that the anchored surface is localized, that is, there exists a constant $C>0$ such that for any $n$, almost all the surface is within distance $C$ of the meridian plane $[0,n]^{d-1} \times \{\tfrac H 2 \}$. In that case, by a similar proof as Theorem \ref{thm:main2}, 
 we can prove that there exists $c>0$ depending on $G$ such that for all $n,H\ge 1$
\[\Var (\tau([0,n]^{d-1}\times\{0\}, H))\ge cn^{d-1}\,.\]
This implies that in high dimensions, we don't expect the variance of the anchored surface to be superconcentrated. This is another hint that minimal surfaces behave very differently as geodesics (of codimension $d-1$) in standard first percolation theory.

\item In the spirit of the easier Theorem \ref{thm:psi}, we may further restrict the 1-Lipschitz functions $\psi$ to be equal to $\tfrac H 2$ along $\partial [0,n]^{d-1}$. The localisation result for uniform such 1-Lipschitz functions proved by Peled in \cite{peled2017high} highly suggests that in high enough dimension, the variance of the associated minimal weight $\Psi_{\mathrm{Lip}}^{\mathrm{anchored}}$ will be $\geq c n^{d-1}$. 
\end{enumerate}
We shall discuss this expected different behaviour further in Proposition \ref{p.anchored} as well as in open question \ref{op.anchored}.

\subsection{Idea of  proof}
$ $ \smallskip

\noindent
{\bf Benjamini-Kalai-Schramm and Talagrand.}
As we mentioned above, a similar theorem was first proved for the study of passage times in first passage percolation by Benjamini--Kalai--Schramm \cite{BKS}. 
A key ingredient of \cite{BKS} which we will also use is Talagrand's inequality \cite{talagrand} (see Theorem \ref{thm:Talagrand}). 
 To obtain a ``sub-surface'' (i.e. $o(n^{d-1})$) upper-bound using Talagrand's inequality, one needs to prove that most edges have a low influence on the maximal flow $\Phi$. In \cite{BKS}, the influence of an edge is related to the probability that the geodesic goes through that edge.  In our setting, it will be related to the probability that the minimal surfaces goes through the plaquette dual to that edge. We refer to  \cite{garban2014noise, duminil2019sharp} for background on the interplay between Boolean functions and statistical physics.
  
\smallskip 
 The main difficulty of this approach, already in \cite{BKS}, is that it happens to be very challenging to upper-bound the influence of any fixed given edge.  In fact, for the passage times in first passage percolation, proving that the maximum influence  in the bulk goes to zero (this is known as the {\em BKS midpoint problem}) was only proved a few years ago by Damron--Hanson \cite{damron2017bigeodesics}, Ahlberg--Hoffman \cite{ahlberghoffman}  and was recently solved quantitatively by Dembin--Elboim--Peled in \cite{DembinElboimPeled}.
\smallskip
 
  To circumvent this, Benjamini--Kalai--Schramm relied in \cite{BKS} on a very nice averaging trick by randomizing the endpoints of the desired passage times. Since the randomized endpoints remain close to the original endpoints of the geodesic, it follows that the difference of passage times between the new geodesic and the original geodesic is negligible compared to the upper bound on standard deviation $\sqrt{n}$. 

\smallskip 
  
\noindent  
{\bf No averaging trick for surfaces.}
 We now explain why this averaging trick fails for surfaces.
 Indeed, consider two surfaces anchored respectively in the boundary of $[0,n]^{d-1}\times \{0\}$ and $[0,n]^{d-1}\times \{1\}$, the best control we can get on the difference of capacity is of order $n^{d-2}$. When $d\ge 3$, we have $n^{d-2}\ge n^{(d-1)/2}$ where $n^{(d-1)/2}$ is the order of the upper bound for the standard deviation for surfaces (obtained for example via Efron-Stein). This shows that as soon as $d\geq 3$, we need to proceed differently as in \cite{BKS} and a close inspection of influences will be needed.
 
\smallskip 
\smallskip

\noindent
{\bf Idea and structure of the proof.}  
We start by noting that if we were considering a maximal flow in a transitive graph, for example the maximal flow with non-trivial homology along the $d^{th}$ direction in a  torus $\mathbb{T}_n^{d-1}\times \mathbb{T}_{H}$, then a direct application of Talagrand's inequality (Theorem \ref{thm:Talagrand}) would readily imply fluctuations of order at most  $n^{\frac {d-1} 2}/\sqrt{\log n}$ for any $H\geq \Omega(n^\epsilon)$ just by using the fact that all edges have the same influence by transitivity of the graph. 

In our present case, despite the lack of transitive action acting on the cylinder $[0,n]^{d-1}\times [0,H]$, the rough idea is that if the minimal surface $\mathcal{E}_n$ (chosen among all possible minimal surfaces in any deterministic way, say) happens to be with high probability at distance at least 1 from the top and bottom boundary, then if we shift vertically by one the set of capacities $\{t(e)\}$  (and also replace the missing bottom capacities by the top capacities that went off the cylinder), one may guess that, again with high probability, the new minimal surface $\mathcal{E}_n(t_{\mathrm{shifted}})$ will be nothing but the vertical shift of $\mathcal{E}_n(t)$. Of course what could prevent this to happen comes from the effect of shuffling the top and bottom capacities.
If one could prove that these two claims indeed happen with high enough probability, then it would imply that all edges in a vertical column have a very close influence which would allow us to conclude using Talagrand's inequality \ref{thm:Talagrand}. 
 
 \smallskip

In the end, we do not quite succeed making this intuition rigorous but our proof is strongly influenced by analysing the effect of such vertical shifts. 
The proof of Theorem \ref{thm:main} will be based on the following three main steps which are of independent interest and do not have an analog in the analysis of Benjamini-Kalai-Schramm in \cite{BKS}:
\begin{enumerate}
\item First, we shall prove that minimal surfaces cannot wiggle too much vertically. This will be achieved in Proposition \ref{lem:insidecyl}. 
A similar phenomenon is known to arise in the analysis of {\em minimal surfaces}, see  \cite{david1998quasiminimal}. Our proof in the discrete setting will rely on the isoperimetric bounds in $\mathbb{Z}^d$ obtained in \cite{bollobas1991edge}. 
This proposition is the technical step which is causing the restriction $h\geq h_0$ in our main theorem. Its proof will be given in Section \ref{s.MS}.

\item Second, we need to know that minimal surfaces are unlikely to stay too close to the top and bottom boundaries. We will not prove this for the true minimal surfaces which lead to the maximal flow $\Phi([0,n]^{d-1}\times \{0\}, H)$ but rather for a slightly modified notion of maximal flow in which minimal surfaces too close to the top and bottom boundaries receive a {\em penalisation}. This modified notion of maximal flow is called $\f$ (see~\eqref{e.f}) and is introduced in Section \ref{s.mainP}. 
For this modified maximal flow $\f$, we can show that the associated minimal surfaces are typically away from the top and bottom boundaries. This is the purpose of Proposition \ref{prop:notX_1}. 
\item Finally, the last difficulty we are facing is the possibility that the minimal surface (for the modified $\f$) may often produce a high vertical cliff at certain locations. This would make the influence profile too inhomogeneous to allow us to control the magnitude of influences.  
Using a deep estimate from Zhang's work \cite{Zhang2017} 
(inspired by the original work by Kesten \cite{Kesten:flows}), we will prove Proposition \ref{prop:ubinfluence} which shows that there are only few edges that may carry a large influence (we believe such edges do not exist but we cannot rule this out rigorously). Its proof will be the purpose of Section \ref{s.inf}. 
\end{enumerate}

\begin{rk}\label{r.easier}
We claim that one can prove Theorem \ref{thm:main2} using the same proof, except there are several drastic simplifications. First, the absence of long thin chimneys (Proposition \ref{lem:insidecyl}) is obvious in this case. Also, vertical cliffs do not exist by definition  (thanks to the 1-Lipschitz condition) and as such Proposition \ref{prop:ubinfluence} is much easier to prove in this case. We leave the details to the reader. 
\end{rk}



\subsection{Background}
\subsubsection*{\bf Definition of maximal flow}\label{s.serious}
 We now provide a more formal definition of maximal flows/minimal surfaces.
 We consider a first passage percolation on the graph $(\sZ ^d,\E ^d)$ where $\E^d$ is the set of edges that link  all the nearest neighbors for the Euclidean norm in $\sZ^d$. Write $(\textbf e_1,\dots,\textbf e_d)$ for the canonical basis of $\sR^d$. We consider a distribution $G$ on $\mathbb R_+$. For each edge $e$ in $\E^d$ we assign a random variable $t_e$ of distribution $G$ such that the family $(t_e)_{e\in\E ^d}$ is independent.


Let $A\subset \sR^{d-1}\times \{0\}$. Let $h>0$, we denote by $\cyl(A,h)$ the cylinder of basis $A$ and height $h$ defined by 
$$\cyl(A,h):=\left\{x+t\textbf e_d\, : \,  x\in A,\, t\in[0,h]\right\}\,.$$
Define the discretized versions $B(A,h)$ and $T(A,h)$ of the bottom and the top of the cylinder $\cyl(A,h)$
$$B(A,h):= \left\{x\in\sZ^d\cap\cyl(A,h)\,:\,\begin{array}{c}
\exists y \notin \cyl(A,h),\, \langle x,y \rangle\in\E^d \\\text{ and $\langle x,y \rangle$ intersects } A
\end{array} \right\}$$
and
$$T(A,h):= \left\{x\in\sZ^d\cap\cyl(A,h)\,:\,\begin{array}{c}
\exists y \notin \cyl(A,h),\, \langle x,y \rangle\in\E^d \\\text{ and $\langle x,y \rangle$ intersects } A+h \textbf e_d
\end{array} \right\}\,.$$

 Let $E\subset \E^d$ be a set of edges. We say that $E$ cuts $B(A,h)$ from $T(A,h)$ in $\cyl(A,h)$ (or is a cutset, for short) if any path from $B(A,h)$ to $T(A,h)$ in $\cyl(A,h)$ intersects $E$. 

We associate with any set of edges $E\subset \E^d$ its capacity $T(E)$ defined by \[T(E):=\sum _{e\in E} t_e\,.\] We define the maximal flow from the top to the bottom of the cylinder $\cyl(A,h)$
\begin{align}\label{e.PHI}
\Phi(A,h):=\min\{T(E)\,:\, E \text{ cuts $T(A,h)$ from $B(A,h)$ in $\cyl(A,h)$}\}\,.
\end{align}
As already mentioned in the introduction, we use the terminology maximal flow as by max-flow min-cut theorem, the dual problem of finding minimal surface boils down to computing the maximal flow. 

\smallskip

From now on, we assume that $G$ can only take two values $0<a<b$. See Open Question \ref{op.GeneralG} for a discussion of possible extensions to more general distributions using for example \cite{BenRos2008,Damron2015}.

\subsubsection*{\bf Dual representation of cutsets}
Let $E\subset \E^d$ be a cutset separating $T(A,h)$ from $B(A,h)$ in $\cyl(A,h)$. The set $E$ is a $(d-1)$-dimensional object, that can be seen as a surface. To better understand this interpretation in term of surfaces, we can associate with each edge $e\in E$ a small plaquette $e^*$. The plaquette $e^*$ is an hypersquare of dimension $d-1$ whose sides have length one and are parallel to the edges of the graphs, such that $e^*$ is normal to $e$ and cuts it in its middle. We also define the dual of a set of edge $E$ by $E^*:=\{e^*,\,e\in E\}$ (see Figure \ref{fig8}). Roughly speaking, if the set of edges $E$ cuts $T(A,h)$ from $B(A,h)$ in $\cyl(A,h)$, the surface of plaquettes $E^*$ disconnects $T(A,h)$ from $B(A,h)$ in $\cyl(A,h)$.  Note that, in dimension $2$, a surface of plaquettes is very similar to a path in the dual graph of $\sZ^2$ and thus the study of minimal cutsets is very similar to the study of geodesics.
\begin{figure}[h]
\def\svgwidth{0.4\textwidth}
\begin{center}
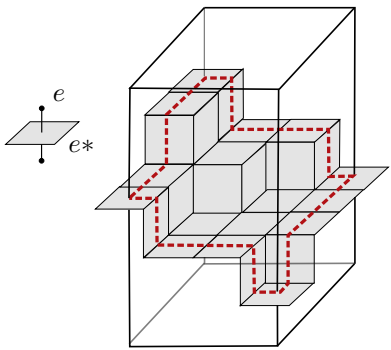
\caption[fig8]{\label{fig8}The dual of a cutset between the top and the bottom of a cylinder for $d=3$.}
\end{center}
\end{figure}

\subsubsection*{\bf Concentration inequalities}

Let $J$ be a finite set of indices. For $\omega\in\{a,b\}^J$ and $j\in J$ denote $\sigma_j \omega$ the function that switches the value in the $j$-th coordinate.
For $f:\{a,b\}^J\rightarrow \sR$, denote
\[\partial _ j f:=\frac{f-f\circ\sigma_j}{2}\,.\]
For $p\in(0,1)$, consider $\mu_p$ the product measure on $\{a,b\}^J$ which gives $a$ with probability $p$ and $b$ with probability $1-p$.
We denote $\|f\|_2^2=\int f^2d\mu_p.$
\begin{thm}[Talagrand's inequality \cite{talagrand} Theorem 1.5]\label{thm:Talagrand}Let $f:\{a,b\}^J\rightarrow \sR$ and $p\in\{0,1\}$.
We have
\begin{equation}
    \Var(f)\le C\log\frac{2}{p(1-p)}\sum_{j\in J}\frac{\|\partial_j f\|_2^2}{1+\log ( \|\partial_j f\|_2/\|\partial_j f\|_1)}
\end{equation}
where $C$ is a universal constant.
\end{thm}
The following proposition is an upper bound on the variance using Efron--Stein inequality.
\begin{thm}[Efron-Stein inequality]\label{thm:efronstein}Let $X=(X_1,\dots,X_n)$ and $X'=(X'_1,\dots,X'_n)$ be two independent and identically distributed vectors taking values in a space $\mathcal X^n$. Let $f:\mathcal X ^n \to \mathbb R$. We have 
\[\Var(f(X))\leq\sum_{i=1}^n \E\left[(f(X)-\E[f(X^{(i)})|X])^2\right]=\sum_{i=1}^n \E\left[(f(X)-f(X^{(i)}))^2_-\right],\]
where $X^{(i)}:=(X_1,\dots,X_{i-1},X'_i,X_{i+1},\dots,X_n)$ and $x_-=\max(-x,0)$.

\end{thm}

\section{Proof of the main theorem}\label{s.mainP}

In this section, we state the main intermediate Propositions which were mentioned in the  Section {\em idea of proof} and which will be proved in the next two Sections. We also implement the {\em penalisation scheme} used to ``localize'' the optimal surface away from the top and bottom boundaries. This will be the purpose of the re-weighting function $Y_{i}$ below.  Finally, using these ingredients we give the proof of  Theorem \ref{thm:main}. 

\subsubsection*{\bf Geometric control on  minimal surfaces.} The proposition stated below will be proved in Section \ref{s.MS}.


\begin{prop}[``Absence of long thin chimneys''] \label{lem:insidecyl}Fix  $0<a<b$. There exists an even $h_0>0$ depending only on $0<a<b$ such that for any $n\ge 1$, $H\ge \frac{1} 2 h_0 n$ and any configuration of capacities in $\{a,b\}$ assigned to the edges of $[0,n]^{d-1}\times [0,H]$, all minimal-cut sets $E$ (i.e. that achieve the infimum in $\Phi([0,n]^{d-1}\times\{0\}, H)$ defined in~\eqref{e.PHI}) are contained in a cylinder of vertical height bounded by $\tfrac1 2 h_0n$. I.e. for any  minimal cut-set $E$, there exists some $u\ge 0$  such that $E\subset [0,n]^{d-1}\times [u,u+ \frac 1 2 h_0n]$.
\end{prop}

\medskip
\noindent
Fix $H\ge h_0 n$. Write $A=[0,n]^{d-1}\times\{0\}$.
Define for $i\le H-\frac {1} 2 h_0n$
\begin{equation}\label{defXi}
    X_i:=\min \left\{T(E): \begin{array}{c} \text{$E$ cuts $B(A+i \mathbf e_d, \frac 1 2 h_0n)$ from $T(A+i \mathbf e_d,\frac 1 2 h_0n)$} \\  \text{in $\cyl (A+i \mathbf e_d,\frac 1 2 h_0n)$}\\\text{ and $E\cap (B(A+i \mathbf e_d,\frac 1 2 h_0n)\cup T(A+i \mathbf e_d,\frac 1 2 h_0n))\ne \emptyset $}\end{array}\right\}\,.
\end{equation}

\subsubsection*{\bf Penalisation scheme.}
Let $0<\ep<\delta<1/4$. Set $M:=\lfloor n^{\ep}\rfloor$ where $\lfloor x\rfloor$ denotes the largest integer smaller than $x$.
 Let $(Z_i)_{1\le i \le M}$ be a family of i.i.d. random variables that takes the value $-1$ with probability $G(\{a\})$ and $1$ with probability $1-G(\{a\})=G(\{b\})$. The reason for this choice is that to apply Talagrand formula (Theorem \ref{thm:Talagrand}) the $t_e$ and $Z_i$ must be parameterized by a Bernoulli random variable with the same parameter.
Set \[S_M:=\sum_{k=1}^M Z_i.\] 
We define
\[\mathfrak{i}_0:=\left\lfloor\frac{H}{2}\right\rfloor+ S_{M}.\]
In particular $\mathfrak{i}_0$ is a  random integer variable taking value in $[\lfloor H/2 \rfloor- M, \lfloor H/2\rfloor + M]$.
We define the family $(Y_i)_{1\le i \le H}$ as follows
\begin{equation}\label{def:Yi}
    \forall 1\le i \le H \qquad Y_i=Y_i(\mathfrak i _0):=\left\{\begin{array}{ll}0&\mbox{if $ |\mathfrak i_0 -i|\le \frac{H}{2} -n ^{\delta}$}\\
\frac{n^{(d-1)/2}}{n^\delta \log n}\left(|\mathfrak i_0 -i|- \frac{H}{2} +n ^{\delta}\right)&\mbox{otherwise.}
\end{array}\right.
\end{equation}
Let $\mathfrak j_0$ be such that 
\[X_{\mathfrak{j}_0}+Y_{\mathfrak{j}_0}=\min_{1\le i\le H - \frac 1 2 h_0n}X_i+Y_i.\]
If there are several possible choices, we pick the smallest. Let $\cE_{min}(\mathfrak{j}_0)$ be the surface achieving the minimum in the definition of $X_{\mathfrak{j}_0}$. Again if there are several possible choices, we choose in a deterministic way (that is invariant by translation along the $\mathbf e_d$ axis).

\subsubsection*{\bf Edges with large influence.} 
The following proposition will be proved in Section \ref{s.inf}. 
\begin{prop}\label{prop:ubinfluence}There exist $n_0=n_0(G)$ and $\xi>0$ such that for all $n\ge n_0$ \[\left|\left\{e\in \cyl(A,H): \Prb(e\in\cE_{min}(\mathfrak{j}_0))\ge n^{-\xi}\right\}\right|\le n^{d-1-\xi}.\]
\end{prop}

\medskip
\noindent
We are  now in position of  proving Theorem \ref{thm:main}.

\begin{proof}[Proof of Theorem \ref{thm:main}] Set $E$ be the set of edges in $\cyl([0,n]^{d-1}\times\{0\}, H) $. Let $\mathrm{I}$ be the set of indices that encode the choice of $\mathfrak{i}_0$, in particular $|\mathrm{I}|=M$.
Set 
\begin{align}\label{e.f}
\f:=\min_{1\le i \le H -\frac 1 2 h_0n }(X_i+Y_i)
\end{align}
where $(X_i)_i$ was defined in \eqref{defXi} and $(Y_i)_i$ in \eqref{def:Yi}.
Thanks to Proposition \ref{lem:insidecyl}, we have 
\[ \Phi([0,n]^{d-1}\times\{0\}, H)=\min_{1\le i \le H -\frac 1 2 h_0n } X_i.\]
It is easy to check that
\[\left|\min_{1\le i\le H -\frac 1 2 h_0n} (X_i+Y_i)-\min_{1\le i\le H -\frac 1 2 h_0n} X_i\right|\le \frac{n^{(d-1)/2}}{\log n}\,.\]
It follows that 
\[\left|\E[\f]-\E[ \Phi([0,n]^{d-1}\times\{0\}, H)]\right|\le \frac{n^{(d-1)/2}}{\log n}\,.\]
and
\begin{equation}\label{eq:varbound}
    \begin{split}
        \Var(\Phi&([0,n]^{d-1}\times\{0\}, H))\\
        &=\E((\Phi([0,n]^{d-1}\times\{0\}, H)-\E \Phi([0,n]^{d-1}\times\{0\}, H))^2)\\&= \E((\Phi([0,n]^{d-1}\times\{0\}, H)-\f +\E \f-\E \Phi([0,n]^{d-1}\times\{0\}, H)+ \f-\E \f)^2)\\&\le 3\left(\Var(\f)+2 \frac{n^{d-1}}{\log n}\right).
    \end{split}
\end{equation}

Let us compute the influence of the bits in $\mathrm{I}$ and $E$.
For $j \in \mathrm{I}$, we have $|\partial_ j  S_M|\le 2$ and it yields that \[|\partial_j \mathfrak i_0|\le 2\qquad\text{and}\qquad    |\partial _j Y_{\mathfrak i_0}|\le \frac{2n^{(d-1)/2}}{n^{\delta}\log  n}.\]
As a result,
\begin{equation*}
    \forall j\in\mathrm{I}\qquad|\partial _j \f|^2\le \frac{4n^{d-1}}{n^{2\delta}\log ^2 n}.
\end{equation*}
Denote $\Delta_e \f=\f\circ\sigma^b_e-\f \circ\sigma^a_e$ where $\sigma^a_e$, $\sigma^b_e$ is the function that changes the value of the edge $e$ to $a$, respectively $b$.
We have
\begin{equation*}
\begin{split}
    \Prb(\partial_e\f\ne0)=\Prb( \Delta_e\f\ne 0)=\frac{1}{G(\{a\})}\Prb( \Delta_e\f\ne 0,t_e=a)&\le \frac{1}{G(\{a\})}\Prb(e\in\cE_{\min}(\mathfrak j_0)).
    \end{split}
\end{equation*}
Note that if $\Delta_e\f\ne 0$ and $t_e=a$, then necessarily $e$ has to belong to the minimal surface.
For $e\in E$, thanks to the previous inequality, we have
\begin{equation*}
    \|\partial_ e\f\|^2_2\le \frac{(b-a)^2}{4}\Prb(\partial_e\f\ne0)\le \frac{(b-a)^2}{4G(\{a\})}\Prb(e\in\cE_{\min}(\mathfrak j_0)).
\end{equation*}
Besides, we have by Cauchy--Schwarz inequality
\begin{equation*}
      \|\partial_ e\f\|_1=\E\left[\left|\partial_e\f\right|\right]\le \sqrt{\Prb(\partial_ e\f\ne 0)}\,\|\partial_ e\f\|_2\le \sqrt{G(\{a\})^{-1}\Prb(e\in\cE_{min}(\mathfrak{j}_0))}\,\|\partial_ e\f\|_2.
\end{equation*}
Let $n_0$ be as in the statement of Proposition \ref{prop:ubinfluence}.
Finally, by applying Theorem \ref{thm:Talagrand} and Proposition \ref{prop:ubinfluence}, we get for $n\ge n_0$
\begin{equation}\label{eq:fin1}
\begin{split}
    & \Var(\f) \\
    & \le C\left(\sum_{j\in \mathrm I}\|\partial_j \f\|_2^2+\sum_{\substack{e\in E:\\\Prb(e\in\cE_{min}(\mathfrak{j}_0))\ge n^{-\xi}}}\|\partial_e\f\|_2^2+\sum_{\substack{e\in E:\\\Prb(e\in\cE_{min}(\mathfrak{j}_0))< n^{-\xi}}}
    \frac{\|\partial_e\f\|_2^2}{1-\log ( G(\{a\})^{-1}\Prb(e\in\cE_{min}(\mathfrak{j}_0))/2}\right)\\
    &\le C\left( |\mathrm{I}|\frac{n^{d-1}}{n^{2\delta}\log ^2 n}+ \frac{(b-a)^2}{G(\{a\})}n^{d-1-\xi}+ \frac{(b-a)^2}{G(\{a\})(1+ \frac\xi 4 \log n)}\sum_{e\in E}\Prb(e\in\cE_{min}(\mathfrak{j}_0)))\right).
\end{split}
\end{equation}
Besides, note that the following set is a cutset from the top to the bottom of the cylinder $\cyl\left(A+\left\lfloor \frac{H}{2}\right\rfloor\mathbf{e}_d, \frac{1}{2}h_0n\right) $
\[\cF:=\left\{\{x,x+\textbf{e}_d\},x \in \left([0,n]^ {d-1}\times\left\{\left\lfloor \frac{H}{2}\right\rfloor\right\}\right)\cap \sZ^ d\right\}.\]
It follows that 
\[\f\le X_{\left\lfloor \frac{H}{2}\right\rfloor}\le b|\cF|=b(n+1)^ {d-1}\]
and
\begin{equation}\label{eq:boundsize}
    a|\cE_{min}(\mathfrak{j}_0)|\le b(n+1)^ {d-1}.
\end{equation}
We conclude by combining inequalities \eqref{eq:varbound}, \eqref{eq:fin1} and \eqref{eq:boundsize}.
\end{proof}

\section{Proof of Proposition \ref{lem:insidecyl} (absence of long chimneys)}\label{s.MS}

We shall need the following discrete isoperimetric inequality
 from \cite{bollobas1991edge} (N.B. the result in \cite{bollobas1991edge} is essentially sharp both in the side-length $n$ and in the dimension $d-1$, we only need the weaker statement given below). 

\begin{thm}[Corollary of Theorem 2 in \cite{bollobas1991edge}]\label{th.Bollo}
For any $d\geq 2$, there exists $c=c(d)>0$ s.t. for any $n\geq 1$ and any set $A\subset [0,n]^{d-1}$, 
\begin{align*}\label{}
|\Delta A| \geq c |A|^{1-\frac 1 {d-1}} \wedge ((n+1)^{d-1}-|A|)^{1-\frac 1 {d-1}}\,,
\end{align*}
where $\Delta A$  stands for the edge boundary of the set $A$ in $[0,n]^{d-1}$ (i.e. $\Delta A:= \big\{ \{i,j\}, \|i-j\|_2=1, i\in A \text{ and } j\in [0,n] ^{d-1} \setminus A \big\}$). 
\end{thm}

\noindent
{\em Proof of Proposition \ref{lem:insidecyl}.} Let $h_0>0$ whose value will be chosen later depending on $a$ and $b$.  Let $H\ge \frac{1} 2 h_0 n$ and let $E\subset \E^ d$ be a cut-set that achieves the infimum in $\Phi([0,n]^{d-1}\times\{0\}, H)$.
\smallskip

Let $h_{max}$ be the maximum height in $\{0,\ldots, H\}$ of a vertex belonging to an edge in the minimal cut-set $E$. Define similarly $h_{min}$. Our goal is then to show that uniformly in the configuration of capacities $\{t(e)\}$, one necessarily has $h_{max}-h_{min} \leq \tfrac {h_0} 2$. 

\subsubsection*{\bf Scanning the upper horizontal slices.}
We start by  scanning the upper horizontal layers of the cut-set $E$ as follows. 
For any $1\leq i \leq h_{max}$, we call the $i^{th}$ upper layer,  $U_i:=[0,n]^{d-1}\times \{h_{max} - i\}$ and we define the following subset of $U_i$. Let $A(i)\subset U_i$ be the set of all points $x\in U_i$ such that any path $\gamma$ connecting $x$ to $[0,n]^{d-1} \times \{H\}$ inside the cylinder $[0,n]^{d-1} \times [h_{max}-i, H]$ necessarily intersects $E$. 

Let us start with the following two easy observations:
\begin{itemize}
\item Since $E$ is a minimal cut-set, it is easy to check that $A(i)\neq \emptyset$ for all $i\geq 1$. 
\item Notice that the edge boundary $\Delta A(i) \subset E \cap U_i$ (N.B. in general, there is no equality). 
\end{itemize}
We will need the following Lemma.

\begin{lem}\label{l.747}
For each $i\geq 1$, let $ F_i:= E \cap [0,n]^{d-1}\times [0,h_{max}-i]$, i.e. the set of all edges in $E$ that belong to the layer $U_i$ or are below that layer. Then for any $i\geq 1$, the set 
\[
E_i: = F_i \cup  \big\{ \{x,x-\mathbf e_d\}, x \in A(i)\big\}
\]
is a cut-set of the cylinder $[0,n]^{d-1} \times [0,H]$. (N.B. Its dual may no longer correspond to a simply connected surface. See Figure \ref{f.art}).  
\end{lem}

\begin{figure}[!htp]
\begin{center}
\includegraphics[width=0.7\textwidth]{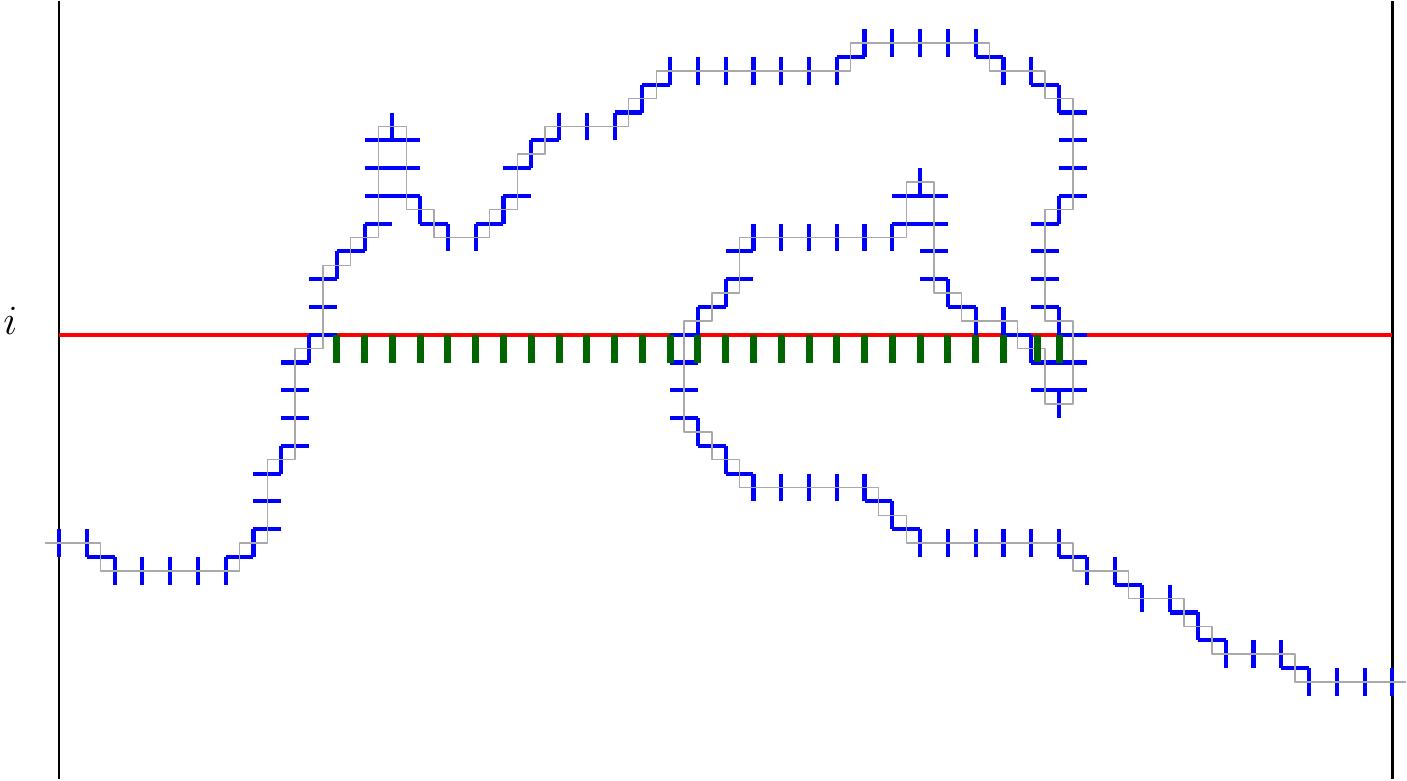}
\end{center}
\caption{Illustration in dimension $d=2(=1+1)$ of the cut-set $E_i$ defined in Lemma \ref{l.747}. It is made here of all the blue edges below level $i$ as well as the additional green edges. By extrapolating such a picture in higher dimension $d\geq 3$, one can easily produce situations where the set $E_i$ splits into distant disconnected parts even though it arises from a minimal cut-set.}\label{f.art}
\end{figure}

\noindent
{\em Proof.} Let $\gamma=\{x_0,x_1,\ldots, x_N\}$ be any connected vertex-path connecting the bottom to the top of the cylinder. Let $1\leq m <N$ be the first time where the path reaches the layer $U_i$, i.e $x_0,\ldots,x_{m-1}$ stays strictly below $L_i$ and $x_m\in U_i$. 
We need to discuss the following two cases: First, if $x_m\in A(i)$, then we are done as the edge $\{x_{m-1}, x_{m}\}$ belongs to $\big\{ \{x,x-\mathbf e_d\}, x \in A(i)\big\}$.  If, on the other hand, the point $x_m\notin A(i)$, then we claim that the path $\{x_0, \ldots, x_m\}$ has necessarily intersected an edge of $F_i$. Indeed, if this was not the case then the path $\{x_0, \ldots, x_m\}$ would arrive at $x_m \notin A(i)$ without ever crossing $E$ and by definition of $A(i)$, one could find a connected continuation of this path $y_1,\ldots, y_M$ such that the path $x_0,\ldots,x_m,y_1,\ldots, y_M$ connects the bottom to the top of the cylinder without ever intersecting the cut-set $E$.  This gives us a contradiction and thus concludes our proof. 
\qed

\smallskip

The reason of this Lemma is that it immediately provides us with the following highly useful constraint: since $E$ is a minimal cut-set and since $F_i \cup \big\{ \{x,x-\mathbf e_d\}, x \in A(i)\big\}$ is a cut-set, we have for all  $i\geq 1$, 
\begin{align}\label{e.const}
a \, |E \setminus F_i| \leq  b | \big\{ \{x,x-\mathbf e_d\}, x \in A(i)\big\}| = b\, |A(i)| 
\end{align}

\medskip
\noindent
We now define
\begin{align}\label{e.stopT}
T:= \min\{ i\geq 1 \text{ s.t. } |A(i)| \geq  (1- \frac {a}{10b}) (n+1)^{d-1}  \}\,.
\end{align}

\smallskip

We shall prove the following Lemma.
\begin{lem}\label{l.tophill}
For any $0<a<b$, there exists $\epsilon=\epsilon(a,b)>0$ s.t.  for any $1\leq k \leq T-1$, 
\begin{align*}
|\Delta  A(k)| \geq \epsilon\, k^{d-2}\,,
\end{align*}

\end{lem}

\noindent
The Lemma is easily proved by induction as follows. Unless $T=1$, the lemma clearly holds for $k=1$. (This is because in this case $\emptyset \subsetneq A(1) \subsetneq [0,n]^{d-1}$).  Now, suppose the Lemma holds for a certain constant $\epsilon>0$ for all $m<k \leq T-1$.

We shall use the above constraint~\eqref{e.const} at the layer $i=k$. Notice that the set of edges $E \setminus F_k$ is by definition the set of edges that are above the layer $k$ (including some vertical edges pointing at that layer). In particular, this set is larger than the set of horizontal edges which lie above the $k^{th}$ layer $U_k$, namely,
\[
E \setminus F_k \supset \bigcup_{m=1}^{k-1} E \cap U_m\,.
\]
Our next crucial point is the fact that for any $m$, as pointed out earlier, one has $\Delta A(m) \subset E \cap U_m$. As such,  this gives us 
\begin{align*}\label{}
|E \setminus F_k |   \geq   \sum_{m=1}^{k-1} |E \cap U_m|  
 &  \geq \sum_{m=1}^{k-1} |\Delta A(m)| \\
& \geq \epsilon \sum_{m=1}^{k-1} m^{d-2} \geq \epsilon \,  C(d) \, k^{d-1}\,. 
\end{align*}
Now plugging this into the constraint~\eqref{e.const} gives us
\begin{align}\label{e.const2}
 b |A(k)|   \geq a  |E \setminus F_k |  \geq  a \epsilon \,  C(d) \, k^{d-1}\,.
\end{align}
Now, using the fact that $|A(k)| < (1- \frac {a}{10b}) (n+1)^{d-1}$ (this is because $k<T$), we obtain from Theorem \ref{th.Bollo} that 
\begin{align*}\label{}
|\Delta A(k)| \geq  c(a,b) |A(k)|^{1- \frac 1 {d-1}}\,.
\end{align*}
(Where for example $c(a,b) = c (\frac a {20b})^{1-\frac 1 {d-1}}$).  Plugging this into~\eqref{e.const2} now gives us 
\begin{align*}\label{}
|\Delta A(k)| \geq c(a,b) \left( \frac {a\epsilon C(d)} b\right)^{1-\frac 1 {d-1}} k^{d-2}\,.
\end{align*}

For $0<a<b$ and the dimension $d$ fixed, one can choose the constant $\epsilon$ small enough so that 
\begin{align*}\label{}
c(a,b) \left( \frac {a\epsilon C(d)} b\right)^{1-\frac 1 {d-1}} > \epsilon\,,
\end{align*}
which ends the proof of the Lemma. \qed 

Now using the Lemma \ref{l.tophill} until $k=T-1$, we extract the following estimate: 
\begin{align*}\label{}
C(d) \epsilon T^{d-1} \leq \sum_{k=1}^{T-1} |\Delta A(k)|  \leq  |E \setminus F_T| \leq  |E| \leq \frac b a (n+1)^{d-1}\,.
\end{align*}
This implies the deterministic statement that the stopping time $T$ is always bounded from above by $\bar h_0\, n$, where $\bar h_0$ is a constant which only depends on $0<a<b$ and the dimension $d$.

\medskip
The rest of the proof will proceed as follows: we will now scan horizontally the cut-set $E$ from its bottom $h_{min}$ and proceed upwards until we reach $h_{min}+T'$.  We will be left with showing that $h_{max}-T$ cannot be much bigger than $h_{min}+T'$.  In order to keep a control on $h_{max}-T$ versus $h_{min}+T'$, it will be important to use exactly the same combinatorial definitions when scanning from below. 

\subsubsection*{\bf Scanning the lower horizontal slices.}
We proceed in the same fashion. 
For any $1\leq i \leq H - h_{min}$, we call the $i^{th}$ lower layer,  $L_i:=[0,n]^{d-1}\times \{h_{min} + i\}$ and we define the following subset of $L_i$. Let $\hat A(i)\subset L_i$ be the set of all points $x\in L_i$ such that any path $\gamma$ connecting $x$ to $[0,n]^{d-1} \times \{H\}$ inside the cylinder $[0,n]^{d-1} \times [h_{max}-i, H]$ necessarily intersects $E$.  (Notice and this is a key point that the set $\hat A(i)$ is nothing but the previous set $A(j)$ with $j=h_{max}-h_{min}-i$).

We will need the following slight adaptation of Lemma \ref{l.747} where we now add additional edges on the top of \textbf{the complement of} $\hat A(i)$. 
\begin{lem}\label{l.748}
For each $i\geq 1$, let $G_i:= E \cap [0,n]^{d-1}\times [h_{min}+i, H]$, i.e. the set of all edges in $E$ that belong to the layer $L_i$ or are above that layer. Then for any $i\geq 1$, the set 
\[
\hat E_i: = G_i \cup  \big\{ \{x,x+\mathbf e_d\}, x \notin \hat A(i)\big\}
\]
is a cut-set of the cylinder $[0,n]^{d-1} \times [0,H]$. 
\end{lem}

\noindent
{\em Proof.} 
Let $\gamma=\{x_0,x_1,\ldots, x_N\}$ be any connected vertex-path connecting the bottom to the top of the cylinder. Let $1 \leq m <N$ be the last passage time of this path through the layer $L_i$. If $x_m \in \hat A(i)$, then by definition of this set, the rest of the connected path $\{x_m,\ldots, x_N\}$ will go through an edge in $G_i$. If on the other hand $x_m \notin \hat A(i)$, then since $x_m$ is the last passage through $L_i$, the next edge is necessarily a vertical edge $\{x_m,x_m\mathbf +  e_d\}$ which  belongs to $\big\{ \{x,x+\mathbf e_d\}, x \notin \hat A(i)\big\}$, this ends the proof.
\qed

Similarly as for the upper layers, we define
\begin{align}\label{e.stopT2}
\hat T:= \min\{ i\geq 1 \text{ s.t. } |(\hat A(i))^c| \geq (1 -\frac{a}{10b}) (n+1)^{d-1}   \}\,.
\end{align}
We claim that the exact same analysis as for the upper layers shows the following two facts:
\begin{enumerate}
\item  for any $1\leq k \leq \hat T-1$, $|\Delta  \hat A(k)| = |\Delta (\hat A(k)^c|  \geq \epsilon\, k^{d-2}$. 
\item $\hat T  \leq \bar h_0 n$. 
\end{enumerate}

\noindent
To conclude our proof, it remains to show that the upper layer where we stop the scanning from above, i.e. $h_{max}- T$ cannot be much higher then the lower layer $h_{min}+\hat T$ at which we stop the scanning from below. In fact, with our choices of stopping times $T$ and $\hat T$, we will show more in the next Lemma, i.e. that up to a safety margin of 1, the top exploration necessarily stops below the bottom exploration. 

\begin{lem}\label{l.intermediate}
\[
h_{min} + \hat T +1   \geq   h_{max} - T\,.
\]
\end{lem}

To prove this Lemma, now that we have analyzed upper and lower horizontal slices, it remains to understand what would happen for the intermediate slices if they were to exist.

\subsubsection*{\bf Scanning the intermediate slices.}

Let us suppose by contradiction that $h_{min}+\hat T + 1 < h_{max}- T$. Introduce \[
M:= h_{max}-T -h_{min} - \hat T \,\,\, (M \geq 2), 
\]
the number of intermediate slices. Let us reparametrize the layers so that 
$i=0$ corresponds to the height $h_{min}+\hat T$ while $i=M$ corresponds to the top intermediate layer $h_{max}-T$. We shall denote by $\{ \tilde A(i) \}_{1\leq i \leq M-1}$ the same sets as before (we use $\tilde A$ instead of $A$ or $\hat A$ just because of the reparametrization). Note that we have  $\tilde A (0) = \hat A(\hat T)$ and $\tilde A(M) = A(T)$. 
\begin{lem}\label{l.intermediate2}
For each $1\leq i \leq M-1$, we have the following 2 constraints.
\begin{enumerate}
\item $a |\tilde A(i)^c| \leq b |A(T)^c| \,\, (\leq \frac a {10} \frac {(n+1)^{d-1}} 2$) 
\item $a | \tilde A(i)|  \leq b |\hat A (\hat T)|  \,\, (\leq  \frac a {10}  \frac {(n+1)^{d-1}} 2$)
\end{enumerate}
\end{lem}
For the inequalities in the parenthesis, we used the definitions of our stopping times $T$ and $\hat T$ (given in~\eqref{e.stopT} and~\eqref{e.stopT2}). Conditions 1) and 2) are incompatible. Therefore this lemma implies that such intermediate layers cannot exist. This implies Lemma \ref{l.intermediate}. To conclude the proof of Proposition \ref{lem:insidecyl}, we are thus left with proving Lemma  \ref{l.intermediate2}.

\smallskip
\noindent
{\em Proof of Lemma  \ref{l.intermediate2}.}
Let us start with item 1. Each point $x$ in the intermediate layer $i$ (i.e. at height $h_{min}+\hat T +i$) which belongs to the set $(\tilde A(i))^c$ has a path in its upper cylinder which connects it to $[0,n]^{d-1}\times \{H\}$ without intersecting $E$. By concatenating this path together with a vertical path pointing down all the way from $x$ to the bottom face $[0,n]^{d-1}\times \{0\}$, since $E$ is a cut-set, it is necessary that at least one edges in this vertical path belongs to $E$. This implies in particular that we have at least $|(\tilde A(i))^c|$ edges of $E$ which are located below layer $i$. Finally, there cannot be too many such edges since $E$ is a minimal cut-set. Using Lemma \ref{l.748} for the layer at height $h_{max}-T$ (or $i=M$), leads us precisely to the constraint 1).

Item 2 is proved in a similar way. For any point $x$ which belongs to $\tilde A(i)$, if we follow the vertical path above $x$ until we reach the top layer $[0,n]^{d-1}\times \{T\}$, then by definition of $\tilde A(i)$, the path will go through at least one edge of $E$. This implies in particular that there are at least $|\tilde A(i)|$ edges in $E$ above (or touching) layer $i$.  Now using Lemma \ref{l.747} for the layer at height $h_{min}+\hat T$ (or $i=0$) together with the fact that $E$ is minimal leads us to constraint 2.  \qed

\begin{rk}\label{}
In the context of minimal surfaces in the continuum setting, a similar phenomenon of absence of ``long thin chimneys" has  been observed for example in  \cite{david1998quasiminimal}.
\end{rk}

\section{Proof of Proposition \ref{prop:ubinfluence}}\label{s.inf}
Let us first prove the following proposition which states that it is unlikely that the minimal surface $\cE_{min}(\mathfrak{j}_0)$ sticks to the bottom or the top of the cylinder.
\begin{prop}\label{prop:notX_1}There exists $n_0=n_0(G)\ge 1$ such that for all $n\ge n_0$,
we have
\[\Prb\left(\mathfrak{j}_0\in\{1,2\}\right)\le \frac{2}{\sqrt n}\,\]
and
\[\Prb\left(\mathfrak{j}_0\in\left\{H-\frac{1}{2}h_0n , H-\frac{1}{2}h_0n -1\right\}\right)\le \frac{2}{\sqrt n}.\]
\end{prop}
To prove this proposition, we will need the following upper bound on the variance.
\begin{prop}[Efron--Stein]\label{prop:Efronstein}There exists a constant $\beta>0$ depending on $G$ such that for all $n\ge 1$ and $H\ge 1$, we have
\[\Var (\Phi([0,n]^{d-1}\times\{0\}, H))\le \beta n^{d-1}\,.\]
\end{prop}
\begin{proof}The proof is a straightforward application of Theorem \ref{thm:efronstein}.
 Let $e_1,\dots,e_N$ be a deterministic ordering of the edges of the cylinder $\cyl([0,n]^{d-1}\times\{0\}, H))$. Set $X=(t_{e_1},\dots,t_{e_n})$ and $f(X)=\Phi([0,n]^{d-1}\times\{0\}, H)$. Let $\cE_{min}$ be a minimal surface for $X$ (chosen according to a deterministic rule in case of ties). Recall that $X^{(i)}$ denotes the vector $X$ where the $i$-th edge has been resampled. Note that if $f(X)<f(X^{(i)})$ then $e_i$ belongs $\cE_{min}$.
 By similar reasoning as in \eqref{eq:boundsize}, we have 
 \[|\cE_{min}|\le \frac b a (n+1)^{d-1}.\]
By applying Theorem \ref{thm:efronstein}, it follows that
 \[\Var(f(X))\le \sum_{i=1}^N (b-a)^2\Prb(e_i\in \cE_{min})\le (b-a)^2\frac b a (n+1)^{d-1}.\]
 This concludes the proof.
\end{proof}

\begin{proof}[Proof of Proposition \ref{prop:notX_1}]
Thanks to Proposition \ref{lem:insidecyl}, we have
\[\Phi([0,n]^{d-1}\times\{0\}, H)=\min_{1\le i \le H- \frac{h_0}2 n}X_i\,.\]
We will just prove the first inequality as the proof for the second inequality is similar.
Let us assume by contradiction that 
\[\Prb\left(\mathfrak{j}_0=1\right)=\Prb\left(\min_{1\le i\le H-\frac{1}{2}h_0n} X_i+Y_i=X_1+ Y_1\right)\ge \frac{1}{\sqrt n}\,.\]
We have for $n$ large enough
\[|\mathfrak i_0 -1|\ge \frac{H}{2}-n^{\ep}-1>\frac{H}{2}-n^{\delta}+ \frac{n^{\delta}}{2}\qquad\text{and}\qquad Y_1\ge \frac{n^{(d-1)/2}}{2\log n}.\]
For all $i\in [2 n ^\delta, 3H/4]$, we have $Y_i=0$.
On the event $\{\min_{1\le i\le H-\frac{1}{2}h_0n} X_i+Y_i=X_1+ Y_1\}$, we have
\[X_1\le \min_{i\in [2 n ^\delta, 3H/4]}X_i- \frac{n^{(d-1)/2}}{2\log n}\,.\]
Hence,
\[\Prb\left(X_1\le \min_{i\in [2 n ^\delta, 3H/4]}X_i- \frac{n^{(d-1)/2}}{2\log n}\right)\ge \frac{1}{\sqrt n}\,.\]
Set 
\[\cE_j:=\left\{ X_j\le \min_{i\in [j+2 n ^\delta, 3H/4]}X_i- \frac{n^{(d-1)/2}}{2\log n}\right\}\,.\]
Since the distribution of  $(X_i)_{1\le i \le 3H/4}$ is the same as the distribution of $(X_i)_{j\le i \le 3H/4+j-1}$, we have
\[\Prb(\cE_j)\ge \frac{1}{\sqrt n}\,.\]
Set for $1\le k \le H/4n^\delta$
\[I_k:= [4kn^\delta, 2(2k+1) n^\delta ]\qquad\text{and}\qquad\cF_k:=\bigcup_{j\in I_k}\cE_j\,.\]
Let $\mathrm N$ be the number of $k\le H/4n^\delta$ such that $\cF_k$ occurs, that is,
\[\mathrm N:=\sum_{1\le k\le H/4n^\delta}\ind_{\cF_k}\,.\]
We have
\begin{equation}\label{eq:espN}
    \E[\mathrm N]\ge \sum_{1\le k\le H/4n^\delta}\Prb(\cF_k)\ge \frac{h_0\sqrt{n}}{4n^\delta}\ge \frac{h_0}{2}n ^{1/4}
\end{equation}
where we recall that $H\ge h_0n$.
Let $i_1<\dots<i_{\rN}$ be integers such that they all belong to different intervals in $(I_k,1\le k\le H/4n^\delta)$ and for all $1\le j\le \rN$, the event $\cE_{i_j}$ occurs.
Note that $i_{j+1}-i_j\ge 2n^\delta$ since they belong to different intervals. Moreover, on the event $\cE_{i_j}$, we have
\[X_{i_j}\le X_{i_{j+1}} - \frac{n^{(d-1)/2}}{2\log n}\,.\]
We can prove by induction that for $0\le k \le \rN-1$
\[X_{i_{\rN-k}}\le \min_{H/2+1\le i\le 3H/4}X_i -(k+1)\frac{n^{(d-1)/2}}{2\log n}.\]
Hence,
\[\min_{1\le i \le H/2}X_i\le \min_{H/2+1\le i \le 3H/4}X_i -\rN\frac{n^{(d-1)/2}}{2\log n}\,.\]
It follows that for $t\ge 0$ using Bienaym\'e--Chebyshev's inequality and Proposition \ref{prop:Efronstein}
\begin{equation*}
    \begin{split}
        \Prb(\rN \ge 2t \log ^2 n)&\le \Prb \left(\min_{H/2+1\le i \le 3H/4}X_i- \min_{1\le i \le H/2}X_i\ge  t n^{(d-1)/2}\right)\\
        &\le 2\frac{\Var( \min_{1\le i \le H/2}X_i)}{t^2 n^{d-1}}\le \frac{2\beta }{t^2}\,.
    \end{split}
\end{equation*}
It yields that 
\[\E(\rN)\le 2(1+2\beta)\log ^2 n\,.\]
This contradicts inequality \eqref{eq:espN} for $n$ large enough depending on $G$.
By the same reasoning we can prove that
\[\Prb\left(\mathfrak{j}_0=2\right)\le \frac{1}{\sqrt n}.\]
This completes the proof.
\end{proof}

To prove Proposition \ref{prop:ubinfluence}, we will also need the following lemma on the regularity of influences under translation by $\mathbf{e}_d$.
\begin{lem}\label{lem;diffreg} There exists $n_0=n_0(G)$ such that for all $n\ge n_0$, $H\ge h_0 n$ the following holds. Let $e$ be an edge of $\cyl(A,H)$ such that $e+2\mathbf e_d \subset \cyl(A,H)$, we have
\[|\Prb(e\in \cE_{min}(\mathfrak{j}_0))-\Prb(e+ 2\mathbf e_d\in \cE_{min}(\mathfrak{j}_0))|\le \frac{2}{n^{\ep/2}} \,.\]
\end{lem}
\begin{proof}[Proof of Lemma \ref{lem;diffreg}]
Let $(t_e)_{e\in \cyl(A,h)}$. We define $t'_e$ as follows
\[t'_e:=\left\{\begin{array}{ll}
t_{e+2\textbf{e}_d}&\mbox{if $e+2\textbf{e}_d\in\cyl(A,H)$}\\
t''_e&\mbox{otherwise}\end{array}\right.\]
where $(t''_e)_{e\in \cyl(A,h)}$ is independent from $(t_e)$.
Let $(Z_i)_{1\le i \le M}$, $(Z'_i)_{1\le i \le M}$ be two independent family of random variables that take the value $-1$ with probability $G(\{a\})$ and $1$ with probability $1-G(\{a\})=G(\{b\})$.
Set \[S_k:=\sum_{k=1}^k Z_i\quad\text{and}\quad S'_k:=\sum_{k=1}^k Z'_i.\] 
Let 
\[\tau:=\inf\{k\in\{1,\dots,M\}:S'_k\ge S_k+2\}\]
where we use the convention $\inf \emptyset=+\infty$.
Finally, we set
\[\mathfrak{i}_0:=\sum_{k=1}^ MZ_k\qquad \text{and}\qquad \mathfrak{i}'_0:=\sum_{k=1}^ {\min(\tau,M)}Z'_k+\sum_{k=\min(\tau,M)+1}^ MZ_k.\]
Denote by $\cE'_{min}(\mathfrak{j}'_0)$ the minimal cutset corresponding to the family $(t'_e)_{e\in\cyl(A,h)}$ and $\mathfrak{i}'_0$. It is easy to check that it has the same law as $\cE_{min}(\mathfrak{j}_0)$.
Moreover, there exists a universal $C>0$ s.t.
\[\Prb(\mathfrak{i}'_0-\mathfrak{i}_0\ne 2)=\Prb(\tau=\infty)=\Prb(\forall k\in\{1,\dots,M\}\quad S_k-S'_k\ge 0)\le \frac{C}{\sqrt{M}}.\]
On the event $\{\mathfrak{i}'_0=\mathfrak{i}_0+2\}\cap \{\mathfrak j_0\notin \{H-\frac{1}{2}h_0n , H-\frac{1}{2}h_0n -1\}\}\cap \{\mathfrak j'_0 \notin \{1,2\}\}$, we have
\[\forall 1\le j \le  H-\frac{1}{2}h_0n -2\qquad  X_{j}(t_e)+Y_j(\mathfrak i _0)=X_{j+2}(t'_e)+Y_{j+2}(\mathfrak i '_0)\]
and  $\cE_{min}(\mathfrak{j}_0)+2\mathbf{e}_d= \cE'_{min}(\mathfrak{j}_0')$. It yields
\begin{equation*}
    \begin{split}
        |\Prb(e\in \cE_{min}&(\mathfrak{j}_0))-\Prb(e+ 2\mathbf e_d\in \cE_{min}(\mathfrak{j}_0))|\\&\le \Prb(\mathfrak{i}'_0-\mathfrak{i}_0\ne 2)+\Prb(\mathfrak{j}_0\in\{1,2\})+\Prb\left(\mathfrak{j}_0\in\left\{H-\frac{1}{2}h_0n , H-\frac{1}{2}h_0n -1\right\}\right).
    \end{split}
\end{equation*}
Finally, by combining the two previous inequalities and using Proposition \ref{prop:notX_1}, it follows that for $n\ge n_0$ (where $n_0$ is as in the statement of Proposition \ref{prop:notX_1})
\begin{equation*}
    |\Prb(e\in \cE_{min}(\mathfrak{j}_0))-\Prb(e+ 2\mathbf {e}_d\in \cE_{min}(\mathfrak{j}_0))|\le \frac{2}{n^{\ep/2}}
\end{equation*}
The result follows.
\end{proof}

\begin{proof}[Proof of Proposition \ref{prop:ubinfluence}]Let $n_0$ be as in the statement of Lemma \ref{lem;diffreg}. Let $n\ge n_0$.
Let $m\ge 1$ that we will choose later depending on $n$. Set $k=\lfloor n/m\rfloor$.
For $\mathrm i=(i_1,\dots,i_{d-1})\in\{1,\dots,k\}^{d-1}$, we define \[A_{\mathrm i}:=\prod_{j=1}^d[(i_j-1)m,i_jm)\times\{0\}\,.\]
We denote by $\mathrm J$ the set of cylinders that contain an edge such that $\Prb(e\in \cE_{min}(\mathfrak{j}_0))\ge n^{-\ep/8}$, that is,
\[\mathrm J:=\left\{{\mathrm i\in\{1,\dots,k\}^{d-1}}:\exists e\in \cyl(A_{\mathrm i},H)\quad \Prb(e\in \cE_{min}(\mathfrak{j}_0))\ge n^{-\ep/8}\right\}.\]
Note that the set $\mathrm J$ is deterministic. By definition, the edges $e\in\cyl(A_i,H)$ for $i\notin\mathrm{ J}$ have a small influence. We need to make sure that there is a negligible number of edges with a large influence in $\cyl(A_i,H)$ for $i\in\mathrm J$. In particular, we need to avoid that the minimal surface has a too large intersection with these cylinders.

Let us first bound the size of $\mathrm J$. Let us assume that there exists $ e\in \cyl(A_{\mathrm i},H)$ such that $\Prb(e\in \cE_{min}(\mathfrak{j}_0))\ge n^{-\ep/8}$. Without loss of generality assume that $e+\sqrt n\mathbf e_d\in\cyl(A_{\mathrm i},H)$. By Proposition \ref{lem;diffreg}, we have
\[|\Prb(e\in \cE_{min}(\mathfrak{j}_0))-\Prb(e+2j\mathbf e_d\in \cE_{min}(\mathfrak{j}_0))|\le \frac{2j}{n^{\ep/2}}\,.\]
Hence, for every $j\le n^{\ep/4}/4$, we have
\[\Prb(e+2j\mathbf e_d\in \cE_{min}(\mathfrak{j}_0))\ge \frac{1}{n^{\ep/8}}-\frac{2j}{n^{\ep/2}}\ge \frac{1}{2n^{\ep/8}} \,.\]
It yields that
\begin{equation*}
    \E[|\cE_{min}(\mathfrak{j}_0)\cap\cyl(A_{\mathrm i},H)|]\ge \frac{n^{\ep/4}}{8n^{\ep/8}}\ge \frac 1 8 n^{\ep/8}\,.
\end{equation*}
Hence, we get using inequality \eqref{eq:boundsize}
\[|\mathrm J| \frac{n^{\ep/8}} 8 \le\sum_{\mathrm i\in \mathrm J}\E[|\cE_{min}(\mathfrak{j}_0)\cap\cyl(A_{\mathrm i},H)|]\le \E[|\cE_{min}(\mathfrak{j}_0)\cap\cyl(A,H)|]\le \frac{b}{a} (n+1)^{d-1},\]
it follows that for some positive constant $\beta$ depending on $a,b$ and $d$
\[|\mathrm J|\le \beta n^{d-1-\ep/8}.\]
Next, we aim at upper bounding the total influence of edges in $\cyl(A_i,H)$ for $i\in\mathrm J$, that is $\E\left[ |\cE_{min}(\mathfrak{j}_0)\cap\cup_{\mathrm i\in \mathrm J}\cyl(A_{\mathrm i},H)|\right] $. 

Let $\cE$ be a cutset in the cylinder $\cyl(A,h)$, one can check that $\cE\cap \cyl(A_{\mathrm i},H)$ is also a cutset from the top to the bottom for the cylinder $\cyl(A_{\mathrm i},H)$.
It follows that
\[\Phi(A_{\mathrm i},H)\le T(\cE\cap \cyl(A_{\mathrm i},H)).\]
Hence, it yields
\[\sum_{\mathrm i\in\{1,\dots,k\}^{d-1}\setminus \mathrm J}\Phi(A_{\mathrm i },H) + a \sum_{\mathrm i\in \mathrm J}|\cE_{min}(\mathfrak{j}_0)\cap\cyl(A_{\mathrm i},H)| \le T(\cE_{min}(\mathfrak{j}_0))\le \Phi(A,H)+ n^{(d-1)/2}.\]
Taking the expectation, we get
\begin{equation}\label{eq:prop61}
    a \E\left[\sum_{\mathrm i\in \mathrm J}|\cE_{min}(\mathfrak{j}_0)\cap\cyl(A_{\mathrm i},H)|\right] \le \E[\Phi(A,H)]-  \sum_{\mathrm i\in\{1,\dots,k\}^{d-1}\setminus \mathrm J}\E[\Phi(A_{\mathrm i },H)]+  n^{(d-1)/2} \,.
\end{equation}
To control the right hand side, we will need a result of Zhang \cite{Zhang2017}.

Let $K= \lceil n/(m-\lfloor m^{5/6}\rfloor)\rceil$.
Set $A':=[0,K(m-\lfloor m^{5/6}\rfloor)]^{d-1}\times \{0\}$ where $K$ was chosen in such a way that $A\subset A'$.
Thanks to the fine study of Zhang \cite[inequality (10.22)]{Zhang2017}, there exists $C>0$ such that we have

\begin{equation}\label{e.Z}
\E[\Phi(A',H)]\le\sum_{\mathrm i\in\{1,\dots,K\}^{d-1}}\E[\Phi(A_{\mathrm i },H)]+C\frac{n^{d-1}}{m^{1/16}}.
\end{equation}

Let us briefly explain how to prove this inequality. Let us assume we could prescribe in each cylinder $\cyl(A_i,H)$ a boundary condition for the minimal surface (that is the trace of the surface on the lateral side) in such a way that these boundary conditions match for adjacent cylinders. In other words, by taking the union of all minimal cutsets in $\cyl(A_i,H)$, $i\in\{1,\dots,k\}^{d-1}$, one would get a cutset in the big cylinder $\cyl(A,H)$ and so $\Phi(A,H)]\le \sum\Phi(A_{\mathrm i },H)$. The issue with this strategy is as follows: in order to prescribe a boundary condition without affecting too much the expectation $\E[\Phi(A_{\mathrm i },H)]$, one needs that the trace of the minimal cutset on the lateral sides is negligible with $n^ {d-1}$.  Since this fact is not known, Zhang overpasses this issue by slightly reducing the dimensions of the cylinder's basis (it accounts for the $m-\lfloor m^{5/6}\rfloor)$): since the total size of the minimal surface is of order $m^ {d-1}$, we can find a smaller cylinder where the trace of the minimal surface on the lateral sides is negligible. Once we can prescribe a given boundary condition, we use the symmetry to prescribe to adjacent cylinders some symmetric matching boundary conditions. The union of all these cutsets form a cutset in the big cylinder. Since the cylinders with prescribed boundary conditions are smaller than the original ones, we need to use a larger $K\ge k$ to be sure that $A\subset A'$.

Let us now explain how we can control the right hand side of~\eqref{eq:prop61} using~\eqref{e.Z} from \cite[inequality (10.22)]{Zhang2017}. The notation $\tau_{min}(k_1,\dots,k_{d-1},m)$ corresponds to $\Phi(\prod_{i=1\dots d}[0,k_i]\times\{0\},m)$. We apply the inequality with $k_1=\dots=k_{d-1}=m$, $w_1=\dots=w_{d-1}=K$, $\delta=1/2$. With these notations, the left hand side in (10.22) is equal to $\E[\Phi(A',H)]$.
Since $A\subset A'$, we have
$\E[\Phi(A,H)]\le \E[\Phi(A',H)]$.
It follows that 
\begin{equation}\label{eq:prop62}
    \begin{split}
        \E[\Phi(A,H)]-  &\sum_{\mathrm i\in\{1,\dots,k\}^{d-1}\setminus \mathrm J}\E[\Phi(A_{\mathrm i },H)]\\
        &\le \E [\Phi(A,H)]-\sum_{\mathrm i\in\{1,\dots,K\}^{d-1}}\E[\Phi(A_{\mathrm i },H)]+ (|\mathrm J|+(K-k)^{d-1}) bm^{d-1}\\
        & \le C\frac{n^{d-1}}{m^{1/16}}+b\beta n^{d-1-\ep/8}m^{d-1}+ b\frac{n^{d-1}}{m^{(d-1)/6}}.
    \end{split}
\end{equation}
Finally, combining \eqref{eq:prop61} and \eqref{eq:prop62}, we get
\begin{equation*}
\begin{split}
    a\,\E\left[ \sum_{\mathrm i\in \mathrm J}|\cE_{min}(\mathfrak{j}_0)\cap\cyl(A_{\mathrm i},H)|\right]
    &\le  \frac{n^{d-1}}{m^{1/16}}+b\beta n^{d-1-\ep/8}m^{d-1}+ n^{(d-1)/2}.
    \end{split}
\end{equation*}
Now choose $m=n^{\ep/(16(d-1))}$. There exists $\xi\le \ep/16$ depending on $\ep$ such that for $n$ large enough
\[\E\left[ \sum_{\mathrm i\in \mathrm J}|\cE_{min}(\mathfrak{j}_0)\cap\cyl(A_{\mathrm i},H)|\right]\le n^{d-1-\xi}.\]
We conclude that
\[\left|\left\{e\in \bigcup_{\mathrm i\in \mathrm J}\cyl(A_{\mathrm i},H): \Prb(e\in\cE_{min}(\mathfrak{j}_0))\ge n^{-\xi/2}\right\}\right|\le n^{d-1-\xi/2}.\]
Since $\xi\le \ep/16$, we have by definition of $\mathrm J$ 
\[\left|\left\{e\in \cyl(A,H): \Prb(e\in\cE_{min}(\mathfrak{j}_0))\ge n^{-\xi/2}\right\}\right|\le n^{d-1-\xi/2}\]
(indeed, in the remaining cylinders, all edges have influence less than $n^{-\ep/8}$ which is smaller than $n^{-\xi/2}$). As such, the result follows.
\end{proof}
\section{Proof of Theorem \ref{thm:main2} and fluctuations of anchored surfaces}\label{s.thm2}

We start with the proof of Theorem \ref{thm:main2} which relies on the martingale decomposition method from Newman--Piza \cite{newman1995divergence}.

\smallskip
\noindent
{\em Proof of Theorem \ref{thm:main2}.}

Let $e_1,\dots,e_N$ be a deterministic ordering of the edges of the cylinder $\cyl([0,n]^{d-1}\times\{0\}, H))$. Denote by $\cF_k$ the $\sigma$-algebra generated by $t_{e_1},\dots,t_{e_k}$.
To simplify the notations, denote $f(t_{e_1},\dots,t_{e_N})=\Phi([0,n]^{d-1}\times\{0\}, H))$.
We have the following martingale decomposition
\[\Var(f)=\sum_{k=1}^ N\E[ (\E(f|\cF_k)-\E(f|\cF_{k-1}))^ 2).\]
Let $(t'_e)$ be an independent family distributed as $(t_e)$ and denote $$t^k:=(t_{e_1},\dots,t_{e_k},t'_{e_{k+1}},\dots,t'_{e_N}),\qquad t^k_a:=(t_{e_1},\dots,t_{e_{k-1}},a,t'_{e_{k+1}},\dots,t'_{e_N})$$ $$\text{and}\qquad t^k_b:=(t_{e_1},\dots,t_{e_{k-1}},b,t'_{e_{k+1}},\dots,t'_{e_N}).$$
In particular, we have \[f(t^k)=(t_{e_k}-a)\mathbf{1}_{f(t^k_b)-f(t^k_a)>0}+f(t^k_a).\]
If $f(t^k_b)-f(t^k_a)>0$ we say that the edge $e_k$ is pivotal.
We can rewrite the expression as follows
\begin{equation*}
    \begin{split}
        \Var(f) =\sum_{k=1}^ N\E[ \E(f(t^k)- f(t^{k-1})|(t_e)_e)^ 2) \, & = \,  \sum_{k=1}^ N\E(\E((t_{e_k}-t'_{e_k})\mathbf{1}_{f(t^k_b)-f(t^k_a)>0} |(t_e)_e)^ 2)\\
        &\ge\Var(t_e)\sum_{k=1}^ N\Prb(f(t^k_b)-f(t^k_a)>0)^2\\
        &\ge \Var(t_e)\sum_{k=1}^ N\Prb(e_k\in\cE_{min},t_{e_k}=b)^2.
    \end{split}
\end{equation*}
When $G(\{b\})>p_c(d)$, there exists $c>0$ such that the number of disjoint paths from the top to the bottom of the cylinder with only edges of time $b$ is at least $cn^ {d-1}$ with high probability (see for instance Theorem 7.68 in \cite{Grimmett99}).
In particular, we have
\[\E[\#\{e\in\cE_{min}, t_e=b\}]\ge c n^ {d-1}.\]
It follows that by Cauchy-Schwarz inequality 
\[\Var(f)\ge\frac{\Var(t_e)}{N}\E[\#\{e\in\cE_{min}, t_e=b\}]^2\ge c_0 \frac{n^{d-1}}{H}\]
where $c_0$ depends on $G$ and $d$.
\qed

\smallskip

The same proof allows us to show that fluctuations for anchored surfaces are not superconcentrated under the following hypothesis \eqref{hyp} of localisation.
For any sequence $(h_n)$ such that $h_n$ goes to infinity with $n$, we have
\begin{equation}\label{hyp}
    \lim_{C\to \infty}\limsup_{n\rightarrow\infty }\frac{1}{n^ {d-1}}\E[\#\{e\in \cE_{\min}:e\notin \{x\in\sR^ d: |x\cdot\mathbf{e}_d-\frac {h_n} {2}|\le C\}]=0\tag{H}
\end{equation}
where $\cE_{min}$ is the minimal cutset for the anchored flow $\tau([0,n]^{d-1}, h_n)$.
\begin{prop}\label{p.anchored}
Under the hypothesis \eqref{hyp}, the variance of the anchored flow $\tau([0,n]^{d-1}, H)$ (defined at the end of the introduction) is in $\Omega(n^{d-1})$.
\end{prop}

\section{Chaoticity of the minimal surface}

Consider the notations of the previous section: $f(t_{e_1},\dots,t_{e_N})=\Phi([0,n]^{d-1}\times\{0\}, H))$. Set $X:=(t_{e_1},\dots,t_{e_N})$. Let $X'$ be an independent vector distributed as $X$. Consider $(U_1,\dots,U_N)$ an i.i.d. family of uniform random variables on $[0,1]$. For any $t\in[0,1]$, we define
\[\forall \,1\le i\le N\quad X^t_i:=\left\{\begin{array}{cc}
    X_i & \mbox{if $U_i\ge t$} \\
    X'_i &\mbox{otherwise.} 
\end{array}\right.\]
Denote by $\mathcal P_t$ the set of pivotal edges for $f(X^ t)$ and by $\mathcal I_t$ the set of edges that are in the intersection of all the minimal surfaces for $f(X^t)$. It is easy to check that $\mathcal I_t \subset \mathcal P_t$.   
Following \cite{chatterjee2014superconcentration}, we obtain the following Corollary of Theorem \ref{thm:main}.
\begin{cor}\label{c.chaotic}There exists a positive constant $C$ such that for any $n\ge1$ and $H\ge h_0 n$ 
\[\forall t\ge 0 \qquad \E[|\mathcal I_0\cap \mathcal I_t|]\le \E[|\mathcal P_0\cap \mathcal P_t|]\le C \frac{n ^ {d-1}}{t\log n\Var(t_e)}.\]
\end{cor}
More precisely, this result follows from the following mild extension of Lemma 3.3 from \cite{tassion2020noise}.
\begin{lem}[Small extension of Lemma 3.3 in \cite{tassion2020noise}]For any $n\ge1$ and $H\ge h_0 n$, we have
\[\Var (\Phi([0,n]^{d-1}\times \{0\}, H ))=\Var(t_e)\int_0 ^ 1\E[|\mathcal P_0\cap \mathcal P_t|]dt\,.\]

Moreover, the function  $t\to\E[|\mathcal P_0\cap \mathcal P_t|]$ is non-increasing.
\end{lem}

\section{Open questions}

\begin{opn}\label{op.anchored}
Prove that anchored maximal flow / minimal surfaces are not superconcentrated in high enough dimension $d$. (Thanks to Proposition \ref{p.anchored}, this boils down to showing that Hypothesis \eqref{hyp} holds). 
\end{opn}


\begin{opn}
Prove superconcentration for maximal flows/minimal surfaces in more general domains, as considered for example in \cite{CerfTheret09geoc,CerfTheret09infc,CerfTheret09supc}. In fact, even extending Theorem \ref{thm:main} to the case of tilted cylinders with a rational slope appears to be challenging as Zhang's inequality from \cite{Zhang2017} relies strongly on symmetry and does not adapt easily to rational directions.
\end{opn}

\begin{opn}\label{op.GeneralG}
In this work, we focused on distributions $G$ taking two values $0<a<b$. It would be interesting to extend this analysis to more general distributions. The  works \cite{BenRos2008,Damron2015}  by Bena\"im--Rossignol and Damron--Hanson--Sosoe, where they extend the study of \cite{BKS} to more general distributions are likely to play a key role here. 
$ $
 
Note that for a continuous distribution $G$, the chaoticity property proved in Corollary \ref{c.chaotic} would be more meaningful as the minimal surface would then be a.s. unique. In particular one would control the true intersection of minimal surfaces before and after noise. 
\end{opn}

\begin{opn}
Our main result, Theorem \ref{thm:main}, only works for thick enough cylinders ($H\geq h_0 n$, for some large enough constant $h_0$). This barrier $h_0$ is there only for technical reasons (coming from Proposition \ref{lem:insidecyl}). Show that the result still holds for any $H\geq \Omega(n^\epsilon)$. 
\end{opn}

\begin{opn}
How do the fluctuations scale with $n$ ? Is there an exponent $\alpha(d)\in (d-2,d-1)$ which describes the variance of $\Phi([0,n]^{d-1}\times\{0\}, H)$ when $H$ is, say, linear in $n$ ?
\end{opn}

\subsection*{Acknowledgments.}
 We wish to thank Itai Benjamini, Guy David, Simon Masnou, Ron Peled and Hugo Vanneuville for  useful discussions. 
The research of B.D is supported by the European Research Council (ERC) under the European Union's Horizon 2020 research and innovation program (grant agreement No
851565). The research of C.G. is supported by the Institut Universitaire de France (IUF) and the French ANR grant ANR-21-CE40-0003.

\bibliographystyle{plain}
\bibliography{biblio}
\end{document}